\documentclass[10.5pt,reqno]{amsart}
\usepackage{amsmath,amsxtra,amssymb,latexsym, amscd,amsthm}
\usepackage[mathscr]{eucal}

\usepackage{color}
\usepackage{graphics} 
\usepackage{epsfig} 
\usepackage{epstopdf}

\newtheorem{thm}{Theorem}[section]
\newtheorem{cor}[thm]{Corollary}
\newtheorem{lem}[thm]{Lemma}
\newtheorem{prop}[thm]{Proposition}
\theoremstyle{definition}
\newtheorem{defn}[thm]{Definition}

\theoremstyle{remark}

\numberwithin{equation}{subsection}

\newcommand{\eps}{\varepsilon}

\newcommand{\N}{\mathbb{N}}

\newcommand{\R}{\mathbb{R}}

\def\eps{\varepsilon}

\def\leq{\leqslant}

\def\geq{\geqslant}

\def\bar{\overline}
\begin{document}
\title[DYNAMICAL BEHAVIOR OF A STOCHASTIC SIRS EPIDEMIC MODEL]{DYNAMICAL BEHAVIOR OF A STOCHASTIC SIRS EPIDEMIC MODEL}

\author{N. T. Hieu $^\natural$} 
\author{N. H. Du$^\dag$} 
\author{P. Auger$^\sharp$} 
\author{N. H. Dang$^\flat$} 
\address {$^{\natural,\sharp}$ UMI 209 IRD UMMISCO, Centre IRD France Nord, 32 avenue Henri Varagnat, 93143 Bondy cedex, France}
\address {$^{\natural,\dag}$ Faculty of Mathematics, Informatics and Mechanics, Vietnam National University, 334 Nguyen Trai road, Hanoi, Vietnam}
\address{$^\natural$ Ecole doctorale Pierre Louis de sant\'e publique, Universit\'e Pierre et Marie Curie, France}
\address{$^\flat$ Department of Mathematics, Wayne State University, Detroit, MI 48202, USA.}
\email{$^\natural$: hieunguyentrong@gmail.com}
\email{$^\dag$: dunh@vnu.edu.vn}
\email{$^\sharp$: pierre.auger@ird.fr}
\email{$^\flat$: dangnh.maths@gmail.com}
\thanks{$^\dag$ corresponding author}
\thanks{$^\dag$ This second author was supported in part by  the Grant NAFOSTED,
N$_0$ {101.03-2014.58. } The fourth author was supported in part by the National Science Foundation DMS-1207667.}

\begin{abstract} 
In this paper we study the Kernack - MacKendrick model under telegraph noise. The telegraph noise switches at random between two SIRS models. We give out conditions for the persistence of the disease and the stability of a disease free equilibrium. We show that the asymptotic behavior highly depends on the value of a threshold $\lambda$ which is calculated from the intensities of switching between environmental states, the total size of the population as well as the parameters of both SIRS systems. According to  the value of $\lambda$, the system can globally tend towards an endemic case or a disease free case. The aim of this work is also to describe completely the omega-limit set of all positive solutions to the model. Moreover, the attraction of the omega-limit set and the stationary distribution of solutions will be pointed out.
\end{abstract}

\maketitle

\section{INTRODUCTION} 

The dynamics of disease spreading among a population have been investigated very widely in the frame of deterministic models e.g. \cite{BC}, \cite{VC}, \cite{LEK}, \cite{JDM}. In such deterministic models, the environment is assumed to be constant.  However, in most real situations, it is necessary to take into account random change of environmental conditions and their effects on the spread of the disease. 
For instance, the disease can be more likely to spread in wet (cold) condition rather than in dry (hot) condition or any other characteristics of the environment that may change randomly. Therefore, it is important to consider the disease dynamics under the impact of randomness of environmental conditions.
There are many papers about this topic in recent years e.g. \cite{AEL}, \cite{GGMP}, \cite{JJS}, \cite{LA}.

{\color{black} Weather conditions can have important effects on the triggering of epidemics. Cold and flu are influenced by humidity and cold temperatures. Viruses are more likely to survive in cold and dry conditions. Hard winds, rain, cold as well as  large variations of temperatures are factors that weaken the immune system.  Lack of sun also provokes a decrease of the level of D vitamin. We could also mention malaria which is influenced by rain and humidity level of air. Weather conditions have in general important effects on epidemics spread. Weather conditions change according to seasons and also to random variations. In the present model, we do not take into account periodic seasonal weather changes. Therefore, we consider an environment which is assumed to be rather constant on average all along the year. In this contribution, we only study the effects of random variations of weather conditions on the spread of epidemics. To simplify our description, we also assume that only two cases can occur, bad or good weather conditions.  Bad weather conditions corresponds to a case where the epidemics is more likely to spread and inversely for good conditions. Therefore, we consider that there exist two models associated with different parameters values corresponding respectively to bad and good weather conditions and that the system switches randomly from the one to the other. For disease models with noise, we also refer to recent contributions \cite{Br},\cite{Ho}.}

The basic simplest epidemic model that we consider is the classical SIRS model introduced by Kernack-MacKendrick of the form (see \cite{LEK} for details)
\begin{equation} \label{e1}
\begin{cases} \dot S = -aSI+cR \\ 
                          \dot I = aSI-bI \\  
                          \dot R=bI-cR,
\end{cases}
\end{equation}
where the susceptible $(S)$, infective $(I)$ and removed $(R)$ classes are three compartments of the total population $N$. Transitions between these compartments are denoted respectively by $a, b,$ and $c$. They describe the course of the transmission, recovery and loss of immunity.

In further studying the SIRS model, we note that the sum $S + I + R=N$ and it is a constant of population size. So that for convenience the removed class $(R)$ can always be eliminated. The reduction  of the equation \eqref{e1} is then
\begin{equation} \label{e2}
\begin{cases} \dot S = -aSI+c(N-S-I) \\ 
                          \dot I = aSI-bI .
\end{cases}
\end{equation}
It is easy to analyze the previous simple system (\ref{e2}) and to show that two situations can occur (see \cite{He}, \cite{LEK}, \cite{Ma}):

- If the basic reproduction number $\mathcal R_0=\frac {Na}{b}>1$ the disease spreads among the population and a positive equilibrium $(s_*, i_*)$ is globally asymptotically stable. It is therefore an endemic situation.

- If $\mathcal R_0=\frac {Na}{b}<1$ the disease is eradicated as a disease free equilibrium $(N,0)$, which is asymptotically stable. This situation is the eradication of the disease among the population.

In this work, we shall concentrate on the switching two classical Kernack and MacKendrick SIRS model, which will be chosen as the basic models for the epidemics. We shall assume that there are two environmental states in each of which the system evolves according to a deterministic differential equation and that the system switches randomly between these two states. Thus, we can suppose there is a telegraph noise affecting on the model in the form of switching between two-element set, $E=\{+,-\}$. With different states, the disease dynamics  are different. The stochastic displacement of environmental conditions provokes model to change from the system in state $+$ to the system in state $-$ and vice versa.

Several questions naturally arise. For instance, in the case where the disease spreads in an environmental condition, while it is vanished in the other one, what will be the global and asymptotic behavior of the system? 
Using the basic reproduction number $\mathcal R_0 $ of both models and the switching intensities, can we make predictions about the asymptotic behavior of the global system, i.e., the existence of a global endemic state or a disease free state?

The paper has 5 sections. Section 2 details the model and gives some properties of the boundary equations. In Section 3, dynamic behavior of the solutions is studied and the $\omega$-limit sets are completely described for each case. It is shown that the threshold $\lambda$ which will be given later plays an important role to determine whether the disease will vanish or be persistent. {\color{black}We also prove the existence of a stationary distribution and provide some of its nice properties. In Section 4, some simulation results illustrate the behavior of the SIRS model under telegraph noise. The conclusion presents a summary of the results and some perspectives of the work. The last section is the appendix where the proofs of some theorems are given.}

\bigskip  
\section{PRELIMINARY} 
\medskip  

Let us consider a continuous-time Markov process $\xi_t, \; t \in R_+$, defined on the probability space $(\Omega, \mathcal F, \mathbb P)$,  with values in the set of two elements, say $E = \{+,-\}$. Suppose that $(\xi_t)$ has the transition intensities $+\overset{\alpha} \to - $ and $-\overset{\beta} \to + $ with $\alpha > 0, \beta >0$. The process $(\xi_t)$ has a unique stationary distribution 
$$ p =\lim_{t \to \infty}\mathbb P\{\xi_t=+\} = \frac {\beta}{\alpha +\beta};\; q=\lim_{t \to \infty}\mathbb P\{\xi_t=-\} = \frac {\alpha }{\alpha +\beta}.$$
The trajectory of $(\xi_t)$ is piecewise constant, cadlag functions. Suppose that
$$0 = \tau_0 < \tau_1 < \tau_2<...< \tau_n < ...$$
are its jump times. Put 
        $$\sigma_{1} = \tau_1 - \tau_0, \; \sigma_{2}= \tau_2 - \tau_1,...,\sigma_{n} = \tau_n - \tau_{n-1}...$$
 It is known that, if $\xi_0$ is given, $(\sigma_{n})$ is a sequence of independent random variables. Moreover, if $\xi_0 = +$ then $\sigma_{2n+1}$ has the exponential density $\alpha 1_{[0,\infty)}\exp(-\alpha t)$ and $\sigma_{2n}$ has the density $\beta 1_{[0,\infty)}\exp(-\beta t)$. Conversely, if $\xi_0 = -$ then $\sigma_{2n}$ has the exponential density $\alpha 1_{[0,\infty)}\exp(-\alpha t)$ and $\sigma_{2n+1}$ has the density $\beta 1_{[0,\infty)}\exp(-\beta t)$ (see \cite[vol. 2, pp. 217]{[Gh-Sk]}). 
Here $1_{[0,\infty)}=1$ for $t\geq 0$ ($=0$ for $t<0$).

In this paper, we consider the Kernack-MacKendrick model under the telegraph noise $\xi_t$ of the form:
\begin{equation} \label{e3}
\begin{cases} \dot S = -a(\xi_t)SI +c(\xi_t)(N-S-I)\\
   \dot I = a(\xi_t)SI - b(\xi_t)I
\end{cases},
\end{equation}
where $g: E=\{+,-\} \to {\R}_+$ for $g = a,b,c$. 
The noise $(\xi_t)$ carries out a switching between two deterministic systems
\begin{equation} \label{e4}
\begin{cases} \dot S = -a(+)SI +c(+)(N-S-I)\\
   \dot I = a(+)SI - b(+)I,
\end{cases}
\end{equation}
and
\begin{equation} \label{e5}
\begin{cases} \dot S = -a(-)SI +c(-)(N-S-I)\\
   \dot I = a(-)SI - b(-)I.
\end{cases}
\end{equation}

Since $(\xi_t)$ takes values in a two-element set $E$, if the solution of $(\ref{e3})$ satisfies equation $(\ref{e4})$ on the interval $(\tau_{n-1}, \tau_n)$, then it must satisfy equation $(\ref{e5})$ on the interval $(\tau_n,\tau_{n+1})$ and vice versa. Therefore, $(S(\tau_n),I(\tau_n))$ is the switching point, that is the terminal point of one state and simultaneously the initial condition of the other. It is known that with positive initial values, solutions to both \eqref{e4} and \eqref{e5} remain nonnegative for all $t\geq0$. Thus, any solution to \eqref{e3} starting in  int$\R_+^2$ exists for all $t\geq0$ and remain nonnegative.

It is easily verified that the systems $(\ref{e4})$ and $(\ref{e5})$ respectively have the equilibrium points
\begin{equation} \label{e6}
(s^\pm_*, i^\pm_*)=\Big(\frac{b(\pm)}{a(\pm)},\frac{c(\pm)(N-\frac{b(\pm)}{a(\pm)})}{b(\pm)+c(\pm)}\Big),
\end{equation}
and their global dynamics depend on these equilibriums. Concretely, if $i^\pm_*>0$ then these positive equilibriums are asymptotically stable, i.e., when $N>\frac{b(\pm)}{a(\pm)}$, $\lim_{t \to \infty}(S^\pm(t),I^\pm(t))$ $=(s^\pm_*, i^\pm_*)$. This is the endemic case, both susceptible and infective classes are together present. On the contrary, if $N\leq\frac{b(\pm)}{a(\pm)}$ then $\lim_{t \to \infty}(S^\pm(t),I^\pm(t))=(N,0)$ and the infective class will disappear. It is called the free case\\


\section{DYNAMICAL BEHAVIOR OF SOLUTIONS}
{\color{black} In this section, we introduce a threshold value $\lambda$ whose sign determines whether the system \eqref{e3} persistent or  the number of infective individuals goes to 0. Moreover, the asymptotic behavior of the solution is described in details.}

For any $(s_0, i_0) \in \mbox{int}\R^2_+$ with $s_0+i_0\leq N$, we denote by $(S(t, s_0, i_0, \omega), I(t, s_0, i_0, \omega))$ the solution of (\ref{e3}) satisfying the initial condition $(S(0, s_0, i_0, \omega), I(0, s_0, i_0, \omega)) = (s_0, i_0)$. For the sake of simplicity, we write $(S(t), I(t))$ for $(S(t, s_0, i_0, \omega), I(t, s_0, i_0,\omega))$ if there is no confusion. A function $f$ defined on $[0, \infty)$ is said to be ultimately bounded above (respectively, ultimately bounded below) by $m$ if $\limsup _{t \to \infty} f(t) < m$ (respectively, $\liminf_{t \to \infty} f(t) > m)$. {\color{black}We also have the following definitions for persistence and permanence
\begin{defn} \label{def1}\

1) System \eqref{e3} is said to be {\it persistent} 
if $\limsup_{t \to \infty}S(t)>0, \; \limsup_{t \to \infty}I(t)>0$ for all  solutions of \eqref{e3}. 

2) In case there exists a positive $\epsilon$ such that  
\begin{align*}\epsilon&\leq \liminf_{t \to \infty}S(t)\leq  \limsup_{t \to \infty}S(t)\leq 1/ \epsilon,\\
\epsilon&\leq \liminf_{t \to \infty}I(t)\leq  \limsup_{t \to \infty}I(t)\leq 1/ \epsilon, 
\end{align*}
we call the system \eqref{e3} is {\it permanent}.
\end{defn}}

It is easy to see that the triangle $\nabla:=\{(s,i): s\geq 0, i\geq 0; s+i\leq N\}$ is invariant for the system \eqref{e3}.  
In the future, without loss of generality, suppose that $\frac{b(+)}{a(+)}\leq \frac{b(-)}{a(-)}$. 

{\color{black} We here define the threshold value which play a key role of determining the persistence of the system \eqref{e3}.
\begin{equation} \label{e8}
\lambda=p\big(a(+)N-b(+)\big)+q\big(a(-)N-b(-)\big).
\end{equation}
In the first part of this section, we prove the following classification
\begin{thm} \label{thrm3.2}\
\begin{enumerate}
\item If $\lambda<0$ then $\lim_{t\to\infty}I(t)=0$  and $\lim_{t\to\infty}S(t)=N$.
\item If $\lambda>0$ the the system \eqref{e3} is persistent. Moreover, if $\frac{b(+)}{a(+)},\frac{b(-)}{a(-)}<N$ then the system is permanent.
\end{enumerate}
\end{thm}
The proof of this theorem is divided into several propositions. Firstly, we have}
\begin{prop} \label{prop5} \
\begin{enumerate}
\item[a)] If $\lambda>0$ then there is a $\delta_1>0$ such that $\limsup_{t\to\infty}I(t)>\delta_1$.
\item[b)] If $\lambda<0$ then $\lim_{t\to\infty}I(t)=0$  and $\lim_{t\to\infty}S(t)=N$.
\end{enumerate}
\end{prop}

\begin{proof}
a) The second equation of the system (\ref{e3}) follows
$$\frac{\ln I(t) - \ln I(0)} t= \frac 1 t \int_0^t (a(\xi_{\bar t})S(\bar t)-b(\xi_{\bar t}))d\bar t.$$
Since $I(t)\leq N$,  $\limsup_{t\to\infty}\frac{\ln I(t) - \ln I(0)} t\leq0$. Therefore, 
 
\begin{multline*}
\limsup_{t\to\infty} \Big( \frac 1 t \int_0^t (a(\xi_{\bar t})N-b(\xi_{\bar t}))d\bar t - \frac 1 t \int_0^t a(\xi_{\bar t})(N-S(\bar t))d\bar t\Big)=\\
\limsup_{t\to\infty} \frac 1 t \int_0^t (a(\xi_{\bar t})S(\bar t)-b(\xi_{\bar t}))d\bar t \leq0.
\end{multline*}
Thus, $$\liminf_{t\to\infty} \frac 1 t \int_0^t (a(\xi_{\bar t})N-b(\xi_{\bar t}))d\bar t \leq \liminf_{t\to\infty} \frac 1 t \int_0^t a(\xi_{\bar t})(N-S(\bar t))d\bar t.$$
Because the process $(\xi_t)$ has a unique stationary distribution 
$\lim_{t \to \infty}\mathbb P\{\xi_t=+\} = p$ and $\lim_{t \to \infty}\mathbb P\{\xi_t=-\} = q$, then by the law of large numbers
$$\lim_{t\to\infty} \frac 1 t \int_0^t (a(\xi_{\bar t})N-b(\xi_{\bar t}))d\bar t=p\big(a(+)N-b(+)\big)+q\big(a(-)N-b(-)\big)=\lambda.$$
Denote $g_{min} = \min(g(+), g(-)), \; g_{max} = \max(g(+), g(-))$ for $g = a, b, c$.
We have
\begin{align} \label{e9}
\begin{split}
\liminf_{t\to\infty} \frac 1 t \int_0^t a_{\max}(N-S(\bar t))d\bar t &\geq\liminf_{t\to\infty} \frac 1 t \int_0^t a(\xi_{\bar t})(N-S(\bar t))d\bar t\\
&\geq\liminf_{t\to\infty} \frac 1 t \int_0^t a(\xi_{\bar t})(N-b(\xi_{\bar t}))d\bar t=\lambda.
\end{split}
\end{align}
On the other hand, from
 $$\dot S(t)= -(a(\xi_t)S(t)+c(\xi_t))I(t)+c(\xi_t)(N-S(t))\geq -(a(\xi_t)N+c(\xi_t))I(t)+c(\xi_t)(N-S(t)),$$
it follows
$$\frac {S(t)-S(0)}t \geq \frac 1t \int_0^t -(a(\xi_{\bar t})N+c(\xi_{\bar t}))I(\bar t)d\bar t + \frac 1t \int_0^t c(\xi_{\bar t})(N-S(\bar t))d\bar t.$$
Since $\lim_{t\to\infty}\frac {S(t)-S(0)}t=0$, 
$$\limsup_{t\to\infty}\Big(\frac 1t \int_0^t -(a(\xi_{\bar t})N+c(\xi_{\bar t}))I(\bar t)d\bar t + \frac 1t \int_0^t c(\xi_{\bar t})(N-S(\bar t))d\bar t\Big)\leq0.$$
Hence,
\begin{align} \label{e10}
\begin{split}
\liminf_{t\to\infty}\frac 1t \int_0^t (a_{\max}N+c_{\max})I(\bar t)d\bar t\geq\liminf_{t\to\infty}\frac 1t \int_0^t (a(\xi_{\bar t})N+c(\xi_{\bar t}))I(\bar t)d\bar t\\
\geq \liminf_{t\to\infty}\frac 1t \int_0^t c(\xi_{\bar t})(N-S(\bar t))d\bar t \geq \liminf_{t\to\infty}\frac 1t \int_0^t c_{\min}(N-S(\bar t))d\bar t.
\end{split}
\end{align} 
Combining (\ref{e9}) and (\ref{e10}), we obtain
\begin{align*}
\begin{split}
\liminf_{t\to\infty}\frac 1t \int_0^t I(\bar t)d\bar t\geq\liminf_{t\to\infty}\frac 1t \int_0^t \frac {c_{\min}}{a_{\max}N+c_{\max}}(N-S(\bar t))d\bar t\\
\geq \frac {c_{\min}}{(a_{\max}N+c_{\max})a_{\max}}\lambda >0.
\end{split}
\end{align*}
This inequality implies that there exists $\delta_1>0$ such that $\limsup_{t\to\infty}I(t)>\delta_1$.\\

b) From the inequality
$$\frac{\dot I(t)}{I(t)}=a(\xi_t)S(t)-b(\xi_t)\leq a(\xi_t)N-b(\xi_t),$$
we have
$$\limsup_{t\to\infty}\frac{\ln I(t) - \ln I(0)} t \leq \limsup_{t\to\infty}\frac 1 t \int_0^t (a(\xi_{\bar t})S(\bar t)-b(\xi_{\bar t}))d\bar t\leq\lambda<0,$$
which implies that $\lim\limits_{t\to\infty}I(t)=0.$
 On the other hand,
 $$\dot S(t)= -a(\xi_t)S(t)I(t)+c(\xi_t)(N-S(t)-I(t))\geq -a_{max}NI(t)+c_{\min}(N-S(t)-I(t)).$$
Thus,
$$S(t) \geq \int_0^t {e^{-c_{\min}(t-\bar t)}(-a_{\max} N  +c_{\min})I(\bar t)d\bar t} +S (0) e^{-c_{\min}t} +c_{\min} N \int_0^t {e^{-c_{\min}(t-\bar t )}d\bar t}. $$
We see that $\lim\limits_{t\to\infty}\int_0^t {e^{-c_{\min}(t-\bar t )}d\bar t}=\frac 1{c_{\min}}$. Further, 
by paying attention that  $\lim\limits_{t\to\infty} I(t) = 0$  we also obtain 
$$\lim_{t\to\infty}\int_0^t {e^{-c_{\min}(t-\bar t)}(-a_{\max} N  +c_{\min})I(\bar t)d\bar t}=0.$$
Hence, $\liminf_{t\to\infty} S(t)\geq N$. Combining $S(t)\leq N$ for all $t>0$ gets $\lim_{t\to\infty} S(t)= N.$
The proof is complete.
\end{proof}

By  definition of $\lambda$ in \eqref{e8} we have the following corollary.
\begin{cor} \label{cor6}
If $\frac{b(+)}{a(+)} \geq N $ then $\lim_{t\to\infty}I(t)=0$ and $\lim_{t\to\infty}S(t)=N$.
\end{cor}

In view of Corollary \ref{cor6}, in the following we suppose that $\frac{b(+)}{a(+)} < N$.
\begin{prop} \label{prop3}
$S(t)$ is ultimately bounded below by $S_{\min}>0$ and there is an invariant set for the system (\ref{e3}), which absorbs all positive solutions. 
\end{prop}
\begin{proof}
Let $S_{\min}$ be chosen such that  
\begin{equation}\label{e7b}
-N a(\pm)S_{\min} +c(\pm)\Big(\frac{b(+)}{2a(+)}-S_{\min}\Big)>m>0,
\end{equation}
and let $A=(S_{\min},0), B=(S_{\min},N-\frac{b(+)}{2a(+)}), C=(\frac{b(+)}{2a(+)},N-\frac{b(+)}{2a(+)}), D=(N,0)$. 
In the interior of the triangle $\nabla$ we have 
$\dot I(t)=a(\xi_t)(S(t)-\frac{b(\xi_t)}{a(\xi_t)})I(t)\leq a(\xi_t)(S(t)-\frac{b(+)}{a(+)})I(t)\leq a(\xi_t)(\frac{b(+)}{2a(+)}-\frac{b(+)}{a(+)})I(t) = -a(\xi_t)\frac{b(+)}{2a(+)}I(t)<0$ for all points lying  above the line $BC$, whereas  $\dot S>m$ for all points that are below the line $BC$ and on the left of $AB$ by \eqref{e7b}. Therefore, it is easy to see that the the quadrangle $ABCD$ is invariant under system \eqref{e3} and all positive solutions ultimately go there. 
\end{proof}

\begin{cor} \label{cor60}
If $\lambda>0$ then the system (\ref{e3}) is persistent.
\end{cor}

\begin{proof}
This result follows immediately from  Propositions \ref{prop5} and \ref{prop3}.
\end{proof}

\begin{prop} \label{prop61}   
$I(t)$ is ultimately bounded below by $I_{\min}>0$ if $\frac{b(-)}{a(-)}<N$. As a result, the system \eqref{e3} is permanent.
\end{prop}
\begin{proof} 
Since  $\frac{b(-)}{a(-)}<N$, we can find an $0<\varepsilon _0<\delta_1$  such that $\min\big\{-a(\pm)si+c(\pm)(N-s-i)>0: 0<s\leq \frac{b(-)}{a(-)}, 0<i\leq\varepsilon_0\big\}>0$.
Then, while  $I(t)\leq\varepsilon_0$ and $S(t)\leq \frac{b(-)}{a(-)}$ we have $\dot S>0$ and $\dfrac{\dot I}{\dot S}=\dfrac{(a(\xi_t)S-b(\xi_t))I}{-a(\xi_t)SI+c(\xi_t)(N-S-I)}>-kI$ where $k$ is some positive number.
Denote by $\gamma$ the piece of the solution curve to the equation $\dfrac{d I}{d S}=-kI$ starting at $(S_{\min}, \varepsilon_0)$
and ending at the intersection point $(\frac{b(-)}{a(-)},\varepsilon_1)$ of this solution curve with the line $s=\frac{b(-)}{a(-)}$ .
Let $G$ be the subdomain  of quadrangle ABCD consisting of all $(s, i)\in ABCD$ lying above the curve $\gamma$ if $s\leq \frac{b(-)}{a(-)}$ and lying above the line $i=\varepsilon_1$ if $\frac{b(-)}{a(-)}\leq s\leq N$.
Obviously,  $G$ is invariant domain because $\dfrac{\dot I}{\dot S}>-kI, \dot S>0$ on $\gamma$ and $\dot I>0$ on the segment $I=\varepsilon_1,\frac{b(-)}{a(-)}\leq S\leq N$. 
Since $\limsup\limits_{t\to\infty}I(t) > \delta_1>\varepsilon_0$ and $(S(t), I(t))$ must eventually enter the quadrangle ABCD,  $(S(t), I(t))$ also eventually enters $G$ which implies that $I(t)$ ultimately bounded below by $I_{\min}=\varepsilon_1$.
\end{proof}

Our task in the next part is to describe the $\omega$-limit sets of the system (\ref{e3}).
 Adapted from the concept in \cite {BCF}, we define the (random) $\omega-$limit set 
of the trajectory starting from an initial value $(s_0, i_0)$ by 
$$\Omega(s_0, i_0,\omega)=\bigcap_{T>0}\overline{\bigcup_{t>T}\big(S(t, s_0, i_0, \omega), I(t, s_0, i_0, \omega)\big)}.$$

 This concept is different from the one in \cite{CF} but it is closest to that of an $\omega-$limit set for a deterministic dynamical system. In the case where $\Omega(s_0, i_0,\omega)$ is a.s. constant, it is  similar to the concept of weak attractor and attractor given in \cite{M, YM}. Although, in general, the $\omega$-limit set in this sense does not  have the invariant property,  this concept is appropriate for our purpose of describing the pathwise asymptotic behavior of the solution with a given initial value.

Let $\pi_t^+(s,i)=(S^+(t,s,i), I^+(t,s,i))$, (resp. $\pi_t^-(s,i) =(S^-(t,s,i),I^-(t,s,i))$) be the solution of (\ref{e4}) (resp. (\ref{e5})) starting in the point $(s,i) \in \mathbb R^2_+$. 

From now on, let us fix an $(s_0,i_0)\in \R_+^2$ and suppose $\lambda>0$. This implies that at least one of the systems \eqref{e4}, \eqref{e5} has a globally asymptotically stable positive equilibrium. Without loss of generality, we  assume the equilibrium point of the system (\ref{e4}) has this property, i.e., $\lim_{t \to \infty}\pi_t^+(s,i)=(s^+_*,i^+_*)\in \text{int}\,\R^2_+$ for any $(s,i)\in \text{int}\,\R^2_+$. Also, suppose that $\xi _0=+$ with probability 1. 

For $\varepsilon>0$ small enough, denote by    $\mathcal U_{\varepsilon}(s, i)$ the $\varepsilon$-neighborhood of $(s,i)$ and by $H_{\varepsilon}\subset \mathbb R^2_+$ the compact set surrounded  by $AB, BC, CD$ and the line $i=\varepsilon.$
 Set 
$$S_n(\omega) = S(\tau_n,s_0,i_0, \omega); I_n(\omega) = I(\tau_n,s_0,i_0,\omega),\; \mathcal F_0^n = \sigma(\tau_k: k \leq n);\; \mathcal F_n^\infty = \sigma(\tau_k-\tau_n: k > n).$$ 
We see that $(S_n,I_n)$ is $\mathcal F_0^n-$ adapted. Moreover,  given $\xi_0$, then $\mathcal   F_0^n$ is independent of $\mathcal F_n^\infty$.
{\color{black} To depict the $\omega$-limit set, we need to obtain a key result that $(S_n, I_n)$ belongs to a suitable compact set infinitely often. In addition, we need to estimate the time a solution to \eqref{e4} starting in a compact set enters a neighborhood of the equilibrium $(s^+_*, i^+_*)$. These results are stated in the following lemmas.} 
\begin{lem}\label{lem7}
 Let  $J\subset \nabla\cap\{S>0, I>0\}$ be a compact set and $(s^+_*, i^+_*) \in J$. Then, for any $\delta _2>0$,  there is a $T_1=T_1(\delta _2)>0$ such that $\pi_t^+(s, i)\in \mathcal U_{\delta _2}(s^+_*, i^+_*)$ for any $t\geq T_1$ and $(s, i)\in J$.
\end{lem}
\begin{proof}
Consider the system (\ref{e4}). Since $(s^+_*, i^+_*)$ is asymptotically stable, we can find a $\bar \delta_2=\bar \delta_2(\delta_2)>0$ such that
$$\pi^+_t\big(\mathcal U_{\bar \delta _2}(s^+_*, i^+_*)\big)\subset \mathcal U_{\delta_2 }(s^+_*, i^+_*)\quad\forall t\geq0.$$
On the one hand, for $(s, i)\in J$,  $\lim_{t\rightarrow \infty}\pi^+_t(s,i)=(s^+_*, i^+_*)$ which implies that there exists a $T_{si}$ satisfying 
\begin{equation*}
\pi^+_t(s,i)\in \mathcal U_{\bar\delta _2/2}(s^+_*, i^+_*)\quad\text{for all $t\geq T_{si}$}.
\end{equation*}
By the continuous dependence of the solutions on the initial conditions, there is a neighborhood $\mathcal U_{si}$ of $(s,i)$ such that
for any $(u,v)\in\mathcal  U_{si}$ we have
\begin{equation*}
\pi^+_{T_{si}}(u,v)\in \mathcal U_{\bar\delta _2}(s^+_*, i^+_*).
\end{equation*}
As a result, $$\pi^+_{t}(u,v)\in \pi^+_{t-T_{si}}\big(\mathcal U_{\bar\delta _2}(s^+_*, i^+_*)\big)\subset \mathcal U_{\delta_2 }(s^+_*, i^+_*)\quad\forall t\geq T_{si}.$$
Since $J$ is compact and the family $\{\mathcal U_{si}:(s,i)\in J\}$ is an open covering of $J$, by  Heine-Borel lemma, there is a finite subfamily, namely $\{\mathcal U_{s_ii_i},i=1,2,...,n\}$, which covers $J$.  Let $T_1=\max_{1\leq i\leq n}\{T_{s_ii_i}\}$. We see that 
if $(s, i)\in J$ then $\pi^+_t(s,i)\in \mathcal U_{\delta _2}(s^+_*, i^+_*)$ for any $t\geq T_1$.
\end{proof}

\begin{lem}\label{lem9} 
There is a compact set $K\in \mbox{int}\,\R^2_+$ such that, with probability 1, there are infinitely many $k=k(\omega)\in\mathbb N$ satisfying $(S_{2k+1},I_{2k+1})\in K.$
\end{lem}
\begin{proof}
In the case $\frac{b(-)}{a(-)}<N$, we can choose $K \equiv G$, that was established in Proposition \ref{prop61} and see that $(S_{2k+1},I_{2k+1})\in K$ for every $k>k_0$.

Suppose that $\frac{b(+)}{a(+)}<N<\frac{b(-)}{a(-)}$. With $\delta_1$ is shown in Proposition \ref{prop5}, by a  similar way to the construction of the set $G$ in the proof of Proposition \ref{prop61}, we construct a curve $\gamma^+$ for the system \eqref{e4}, with the initial point $(S_{\min},\delta_1)$ and the end point $(\frac{b(+)}{a(+)},I^+_{\min})$; then define the subdomain $K$ of quadrangle ABCD consisting of every $(s, i)\in ABCD$ lying above the curve $\gamma^+$ if $s\leq \frac{b(+)}{a(+)}$ and lying above the line $i=I^+_{\min}$ if $\frac{b(+)}{a(+)}\leq s\leq N$. It is seen that $K$ is an invariant set for the system \eqref{e4}. 
By Proposition  \ref{prop5},  there is a sequence $(\mu_n)\uparrow\infty$ such that $I(\mu_n)\geq \delta_1$ for all $n\in \N$. Further, since $I(t)$ is decreasing whenever $S(t)\leq\frac{b(+)}{a(+)}$, we can chose $(\mu_n)$ such that $S(\mu_n)>\frac{b(+)}{a(+)}$. If  $\tau_{2k}\leq \mu_n< \tau_{2k+1}$, i.e. $\xi_{\mu_n}=+$, we have $(S_{2k+1},I_{2k+1})\in K$ because $(S(\mu_n),T(\mu_n))\in K$ and $K$ is invariant set for the system \eqref{e4}. For   $\tau_{2k-1}\leq \mu_n< \tau_{2k}$, if $\sup_{\tau_{2k}\leq t< \tau_{2k+1}} I({t})\geq\delta_1$ we return to the above case. Otherwise, $I({t})<\delta_1$ for any $ \tau_{2k}\leq t< \tau_{2k+2}$ since $I(t)$ is decreasing on $[\tau_{2k+1}, \tau_{2k+2})$... Continuing this process, we can either find an odd number $2m+1>k$ such that $(S_{2m+1}, I_{2m+1})\in K$ or see that $I(t)<\delta_1\; \forall t>\tau_{2k}$. The latter contradicts to  Proposition \ref{prop5}. The proof is complete. 
\end{proof}



{\color{black}Having the above lemmas, we now in the position to describe the pathwise dynamic behavior of the solutions of the system (\ref{e3}).}
Put
\begin{equation}\label{e11}
\Gamma=\left\{(s, i)=\pi_{t_n}^{\varrho(n)}\circ\cdots\circ\pi_{t_2}^{+}\circ\pi_{t_1}^{-}( s^+_*,  i^+_*):0\leq t_1,t_2,\cdots,t_n; \; n\in \mathbb N \right\}.
\end{equation}
where $\varrho(k)=(-1)^k$.
 We state the following theorem which is proved in Appendix.
\begin{thm}\label{thm1} \
If $\lambda>0$ then for almost all $\omega$, the closure $\overline {\Gamma}$ of $\Gamma$ is a subset of $\Omega(s_0, i_0,\omega)$. 
\end{thm}

{\color{black}The following theorem provide a complete description of the $\omega$-limit set of the solution to \eqref{e3} in case $\lambda>0$.}
\begin{thm}\label{thm2} Suppose $\lambda>0,$\
\begin{enumerate} 
\item[a)] If  \begin{equation} \label{e12a}
\frac{a(+)}{a(-)}=\frac{b(+)}{b(-)}=\frac{c(+)}{c(-)},
\end{equation} 
the systems \eqref{e4} and \eqref{e5} have the same equilibrium. Moreover, all positive solutions to the system \eqref{e3} converge to this equilibrium with probability 1.
\item[b)] If \eqref{e12a} is not satisfied then, with probability 1, the  $\overline{\Gamma}=\Omega(s_0, i_0,\omega)$. Moreover, $\overline \Gamma$ absorbs all positive solutions in the sense that for any initial value $(s_0, i_0)\in\rm int\R^2_+$, the value  $\gamma{(\omega)}=\inf\{t>0: (S(\bar t, s_0, i_0,\omega), I(\bar t, s_0,i_0,\omega))\in{\bar \Gamma} \,\, \forall\, \bar t>t\}$ is finite outside a $\mathbb P$-null set.
\end{enumerate}
\end{thm}
\begin{proof}
a) It is easy to see that the systems \eqref{e4} and \eqref{e5} have the same equilibrium, $(s^+_*, i^+_*)=(s^-_*, i^-_*)=:(s_*,i_*)$ if and only if the condition \eqref{e12a} is satisfied. 
Let $\varepsilon>0$ be arbitrary. Since $(s_*, i_*)$ is globally asymptotically stable,  there is a neighborhood $ V_\varepsilon\subset \mathcal U_\eps(s_*, i_*)$,  invariant under the system \eqref{e4} (see The Stable Manifold Theorem, \cite[pp 107]{Pe}). Under the condition \eqref{e12a}, the vector fields of both systems \eqref{e4} and \eqref{e5} have the same direction at every point $(s,i)$. As a result, $ V_\varepsilon$  is also invariant under the system \eqref{e5}, which implies that  $ V_\varepsilon$ is invariant under the system \eqref{e3}.
By Theorem \ref{thm1}, $(s_*, i_*)\in \Omega(s_0,i_0,\omega)$ for almost all $\omega$. Therefore, $T_{V_\varepsilon}=\inf\big\{t>0: (S(t), I(t))\in V_{\varepsilon}\big\}<\infty\,$ a.s.
 Consequently, $(S(t), I(t))\in  V_{\varepsilon} \,\, \forall t>T_{ V_\varepsilon}$. This property says that $(S(t), I(t))$ converges to $(s_*, i_*)$ with probability 1 if $S(0)>0, I(0)>0$.

b) We will show that if there exists a $t_0>0$ such that the point $(\bar s_0, \bar i_0)=\pi_{t_0}^-(s^+_*, i^+_*)$ satisfies the following condition
 \begin{equation}\label{e12}
\det\begin{pmatrix} \dot{S}^+(\bar s_0, \bar i_0)& \dot{S}^-(\bar s_0, \bar i_0) \\ \dot{I}^+(\bar s_0, \bar i_0) & \dot{I}^-(\bar s_0, \bar i_0)
\end{pmatrix}\ne 0,
\end{equation}
then, with probability 1, the closure $\overline{\Gamma}$ of $\Gamma$ coincides $\Omega(s_0, i_0,\omega)$ and $\overline \Gamma$ absorbs all positive solutions.

Indeed, let $(\bar s_0, \bar i_0)=\pi_{t_0}^-(s_*^+, i_*^+)$ be a point in ${\rm int}\R_+^2$ satisfying the condition \eqref{e12}. By the existence and continuous dependence on the initial values of the solutions, there exist two numbers  $d>0$ and $e>0$ such that the function $\varphi(t_1, t_2)=\pi_{t_2}^+\pi_{t_1}^-(\bar s_0, \bar i_0)$ is defined and continuously differentiable in $(-d, d)\times (-e, e)$.
We note that 
\begin{align*}
&\det\left(\dfrac{\partial\varphi}{\partial t_1}, \dfrac{\partial\varphi}{\partial t_2}\right)\Big| _{(0, 0)}=\\
&\det\begin{pmatrix} -a(+)\bar s_0\bar i_0 +c(+)(N-\bar s_0-\bar i_0)& -a(-)\bar s_0\bar i_0 +c(-)(N-\bar s_0-\bar i_0)\\ a(+)\bar s_0\bar i_0 - b(+)\bar i_0& a(-)\bar s_0\bar i_0 - b(-) \bar i_0
\end{pmatrix}\ne 0.
\end{align*}
Therefore, by the Inverse Function Theorem, there exist $0<d_1<d,  0<e_1<e$ such that $\varphi(t_1, t_2)$ is a diffeomorphism between  $V=(0, d_1)\times(0, e_1)$ and $U=\varphi(V)$. As a consequence, $U$ is an open set. Moreover, for every $(s, i)\in U$, there exists a  $(t_1^*, t_2^*)\in(0, d_1)\times(0, e_1)$ such that $(s, i)=\pi_{t_2^*}^+\pi_{t_1^*}^-(\bar s_0,\bar i_0)\in S$.  Hence,  $U\subset \Gamma \subset \Omega(s_0,i_0,\omega)$. Thus, there is a stopping time  $ \gamma<\infty$ a.s. such that $(S(\gamma), I(\gamma))\in U$. Since $\Gamma$ is a forward invariant set and $U\subset \Gamma$, it follows that $(S(t),I(t))\in \Gamma\,\,\forall t>\gamma$ with probability 1.  The fact $(S(t), I(t))\in \Gamma$ for all $t>\gamma$ implies that $\Omega(s_0, i_0,\omega)\subset \overline{\Gamma}$. By combining with Theorem \ref{thm1} we get $\overline \Gamma=\Omega(s_0,i_0,\omega)$ a.s.

Finally, we need to show that when condition \eqref{e12a} does not happen, it
ensures the existence of $t_0$ satisfying \eqref{e12}. Indeed, 
consider the set of all $(s,i)\in \text{int\,}\R_+^2$ satisfying
\begin{equation}\label{e12c}
\det\begin{pmatrix} \dot{S}^+(s,i)& \dot{S}^-(s,i) \\ \dot{I}^+(s,i) & \dot{I}^-(s,i)
\end{pmatrix}= 0, 
\end{equation}
or
\begin{equation}\label{e12d}
\big[ -a(+)si+c(+)(N-s-i) \big]\big[a(-)s-b(-)\big]-[ -a(-)si+c(-)(N-s-i) \big]\big[a(+)s-b(+)\big]=0. 
\end{equation}
The equation \eqref{e12d} describes a
quadratic curve. However, it is easy to prove that any quadratic curve is not the integral curve of the system \eqref{e5}. This means that we can find a $t_0$ such that the point $(\bar s_0, \bar i_0)=\pi_{t_0}^-(s^+_*, i^+_*)$ satisfies the condition \eqref{e12}. The proof is complete.
\end{proof}
{\color{black}It is well-known that the triplet $(\xi_t, S(t), I(t))$ is a homogeneous Markov process with the state space 
$\mathscr{V}:=E\times \text{int}\R_+^2$. In the rest of this section, we prove the existence of a stationary distribution for the process $(\xi_t, S(t), I(t))$. Moreover, some nice properties of the stationary distribution and the convergence in total variation are given. If $\lambda>0$ and $\eqref{e12a}$ holds, all positive solutions converge almost surely to the equilibrium $(s_*^+,i_*^+)=(s_*^-,i_*^-)$. Otherwise, we have the following theorem whose proof is given in Appendix.}
\begin{thm}\label{thm3}
If $\lambda>0$ and  \eqref{e12a} does not hold, then $(\xi_t, S(t), I(t))$ has a stationary distribution $\nu^*$, concentrated on $E\times(\nabla\cap \text{int} \,\R^2_+)$. In addition, $\nu^*$ is the unique stationary distribution having the density $f^*$, and for any initial distribution, the distribution of $(\xi_t, S(t), I(t)$ converges to $\nu^*$ in total variation.
\end{thm}
\section{DISCUSSION}

The basic reproduction number $\mathcal R_0$ is an important concept in epidemiology. $\mathcal R_0$ is the threshold parameter for many epidemiological models, it informs whether the disease becomes extinct or whether the disease is endemic. For example, there are many recent papers about periodic epidemic models that concentrate on defining and computing $\mathcal R_0$ (see \cite{BA}, \cite{BG}, \cite{Fr}, \cite {WA} \cite{ZY}). In the classic SIRS model \eqref{e1}, $\mathcal R_0$ is valued by ratio $\frac{Na}{b}$, it represents the rate of increase of new infections generated by a single infectious individual in a total sane population. Based on this $\mathcal R_0$, we give out the key parameter $\lambda$ for our stochastic SIRS model. $\lambda$ reads as follows:
$$\lambda = p\big(a(+)N - b(+)\big) + q\big(a(-)N - b(-)\big).$$
This is the average of two terms associated with each system $+$ or $-$ weighted by the switching intensities. Therefore, in the stochastic model, $\lambda$ can be interpreted as the average number of infective individuals generated by a single infectious individual in a totally sane population for the total system with random switches. We can, therefore, understand that when $\lambda$ is positive, it signifies that asymptotically the total system will go towards an endemic state while the disease will vanish provided that it is negative. Hence, for the stochastic model, $\lambda$ is a very important parameter that enables us to obtain important informations about the asymptotic behavior of the total system. {\color{black} It should also be mentioned that in \cite{GGMP}, when studying the simpler SIS model with Markov switching, the authors also provided a similar threshold value $T^S_0$ for almost sure extinction or persistence. In case of persistence, using the ergodicity of the Markov chain, some ultimate estimates for the solution were given. However, in order to study the SIRS model, we need to consider a system of two differential equations rather than only one differential equation as for the SIS model, so the method used in \cite{GGMP} may not appropriate for the SISR model with Markovian switching. Moreover, that method does not provide a complete description of the asymptotic behavior of the solution. For this reason, we treat our model in another way. To be more precisely, we analyzed the pathwise solution and  obtained a better result. Moreover, some nice properties of the stationary distribution and the convergence of the transition probability were given.}

{\color{black} In this work, we have shown that the asymptotic behavior of the system depends on the sign of a parameter $\lambda$.
We have shown that when $\lambda < 0$, the epidemics always vanishes and
that when $\lambda > 0$, the system is persistent leading to an endemic state.
Furthermore, under some supplementary constraint on parameters, the system is permanent, theorem \ref{thrm3.2}.
When $\lambda >0$, we also have shown properties of stationary distributions. We also illustrated our results by numerical simulations.}

We illustrate different situations in the following numerical simulations. Examples I and II show cases where $\lambda$ is positive, the first one illustrates the switching between two endemic systems $+$ and $-$, whilst the second one depicts a system composed of an endemic case $+$ and a case $-$ for which the disease free equilibrium is stable. In both examples, the simulations show that asymptotically the total system persists leading to an endemic situation.

The last example III considers the case of $\lambda < 0$, with an endemic system $+$ and for the other one $-$ a stable disease free equilibrium. As expected, the simulation shows that after several switches, the disease is globally eradicated.

Examples II and III are interesting because they illustrate a similar case, i.e. when systems $+$ and $-$ have opposite trends, system $+$ being endemic and system $-$ being disease free. In those examples, it is thus questionable to predict what will be the global evolution of the complete system switching at random between these two different situations. The answer is given by looking at the sign of parameter $\lambda$ which allow us to predict if the disease will globally invade or vanish in the long term.

Global changes may have important consequences on the spreading of emergent diseases and epidemics. Therefore, it is important to provide pertinent tools that allow us to make suitable predictions about the possibility of emergence of a disease in a changing environment undergoing climatic and environmental changes. The aim of this paper was to provide such efficient tools.

As a perspective, the system would be extended to the case of a system switching randomly between $n$ states, $n>2$. It would also be interesting to test the model on real situations, like malaria, switching between wet and dry periods. Otherwise, for further study on the epidemic models under the effect of random environmental conditions, we could add some more other stochastic factors as in \cite{BDW} to this SIRS model. {\color{black} An interesting perspective would also be to combine in the same model, periodic seasonal \cite{BA},\cite{BG} and random changes of weather conditions.}
\section{Appendix}
\begin{proof}[The proof of Theorem \ref{thm1}] 
With $K$ is mentioned in Lemma \ref{lem9}, we construct a sequence
\begin{eqnarray*}
\eta_1&=&\inf\{2k+1: (S_{2k+1}, I_{2k+1})\in K\}\\
\eta_2&=&\inf\{2k+1>\eta_1: (S_{2k+1}, I_{2k+1})\in K\}\\
\cdots&&\\
\eta_n&=&\inf\{2k+1>\eta_{n-1}: (S_{2k+1}, I_{2k+1})\in K\}\ldots
\end{eqnarray*}
It is easy to see that $\{\eta_k=n\}\in\mathscr{F}_0^n$ for any $k, n$. Thus the event $\{\eta_k=n\}$ is independent of $\mathscr{F}_n^{\infty}$ if $\xi_0$ is given. By Lemma \ref{lem9}, $\eta_n<\infty$ a.s. for all $n$.

Let $T_2>0, \overline T_2>0$. For any $k\in \mathbb N$, put  $A_k=\{\sigma_{\eta_k+1}<T_2, \sigma_{\eta_k+2} > \overline T_2\}$. 
We have
\begin{align*}
\mathbb P({A_k}) &=\mathbb P\{\sigma_{\eta_k+1}< T_2, \sigma_{\eta_k+2}> \overline T_2\}\\
&=\sum_{n=0}^{\infty }\mathbb P\{\sigma_{\eta_k+1}< T_2, \sigma_{\eta_k+2}> \overline T_2\mid \eta _k=2n+1\}\mathbb P\{\eta _k=2n+1\}\\
&=\sum_{n=0}^{\infty }\mathbb P\{\sigma_{2n+2}< T_2, \sigma_{2n+3}> \overline T_2\mid \eta _k=2n+1\}\mathbb P\{\eta _k=2n+1\}
\\
&=\sum_{n=0}^{\infty }\mathbb P\{\sigma_{2n+2}< T_2, \sigma_{2n+3}> \overline T_2\}\mathbb P\{\eta _k=2n+1\}\\
&=\sum_{n=0}^{\infty }\mathbb P\{\sigma _{2}< T_2, \sigma_3> \overline T_2\}\mathbb P\{\eta _k=2n+1\}=\mathbb P\{\sigma _{2}< T_2, \sigma_3> \overline T_2\}>0.
\end{align*}
Similarly,
\begin{align*}
& \mathbb P({A_k}\cap {A_{k+1}})=\mathbb P\{\sigma_{\eta_k+1}< T_2, \sigma_{\eta_k+2}> \overline T_2, \sigma_{\eta_{k+1}+1}< T_2, \sigma_{\eta_{k+1}+2}> \overline T_2\}\\
&=\sum_{0\leq l<n<\infty }\mathbb P\{\sigma_{\eta_k+1}< T_2, \sigma_{\eta_k+2}> \overline T_2, \sigma_{\eta_{k+1}+1}< T_2, \sigma_{\eta_{k+1}+2}> \overline T_2\mid\\
& \qquad\qquad\eta _k=2l+1,\eta _{k+1}=2n+1 \}\mathbb P\{\eta _k=2l+1,\eta _{k+1}=2n+1\}
\\
&=\sum_{0\leq l<n<\infty }\mathbb P\{\sigma_{2l+2}< T_2, \sigma_{2l+3}> \overline T_2, \sigma_{2n+2}<T_2, \sigma_{2n+3}>\overline T_2\mid \eta _k=2l+1,\\ 
&\qquad\qquad\eta _{k+1}=2n+1 \}\times\mathbb P\{\eta _k=2l+1,\eta _{k+1}=2n+1\} \\
&=\sum_{0\leq l<n<\infty }\mathbb P\{\sigma_{2n+2}< T_2, \sigma_{2n+3}> \overline T_2\}\mathbb P\{\sigma_{2l+2}<T_2, \sigma_{2l+3}>\overline T_2\mid\\
&\qquad\qquad\eta_k=2l+1, \eta _{k+1}=2n+1 \}\times \mathbb P\{\eta _k=2l+1,\eta _{k+1}=2n+1\} 
\\
&=\sum_{0\leq l<n<\infty }\mathbb P\{\sigma_{2}< T_2, \sigma_{3}> \overline T_2\}\mathbb P\{\sigma_{2l+2}< T_2, \sigma_{2l+3}> \overline T_2\mid\\
&\qquad\qquad\eta_k=2l+1, \eta _{k+1}=2n+1 \}\times \mathbb P\{\eta =2l+1,\eta _{k+1}=2n+1\} 
\\
&=\mathbb P\{\sigma_{2}< T_2, \sigma_{3}> \overline T_2\}\sum_{0\leq l<n<\infty }\mathbb P\{\sigma_{2l+2}< T_2, \sigma_{2l+3}> \overline T_2\mid\\
&\qquad\qquad\eta_k=2l+1, \eta _{k+1}=2n+1 \}\times \mathbb P\{\eta _k=2l+1,\eta _{k+1}=2n+1\}
\\ 
&=\mathbb P\{\sigma_{2}< T_2, \sigma_{3}> \overline T_2\}\sum_{l=0}^{\infty }\mathbb P\{\sigma_{2l+2}< T_2, \sigma_{2l+3}> \overline T_2\mid \eta _k=2l+1\}\mathbb P\{\eta _k=2l+1\}\\
&=\mathbb P\{\sigma _{2}< T_2, \sigma_3> \overline T_2\}^2.
\end{align*}
Therefore, $$\mathbb P(A_k\cup A_{k+1})=1-(1-\mathbb P\{\sigma_2<T_2, \sigma_3>\overline T_2\})^2.$$
Continuing this way we obtain
$$\mathbb P\biggl(\bigcup_{i=k}^nA_i\biggl)=1-(1-\mathbb P\{\sigma_2<T_2, \sigma_3>\overline T_2\})^{n-k+1}.$$
Hence, 
\begin{equation}\label{e3.14}
\mathbb P\biggl(\bigcap_{k=1}^{\infty}\bigcup_{i=k}^{\infty}A_i\biggl)={\mathbb P}\{\omega: \sigma_{\eta_n+1}<T_2, \sigma_{\eta_n+2}>\overline T_2 \mbox{ i.o. of }\;n\}=1.
\end{equation}
Fix $T_2>0$. From $\dot I(t)=a(\xi_t)S(t)I(t)-b(\xi_t)I(t)\geq -b_{\max}I(t)$ and $I(\tau_{\eta_k})\geq I_{\min}$, it follows that $I(t+\tau_{\eta_k})\geq I_{\min} e^{-b_{\max}t}$ for all $t>0$. As a result, with $\sigma_{\eta_k+1}<T_2$, $ I_{\eta_k+1}>\Delta:=I_{\min} e^{-b_{\max}T_2}$.

Let $\delta_2>0$, we choose $\bar T_2=T_1(\delta_2)$ as in Lemma \ref{lem7} for the set  $J= H_{\Delta}$. Because $I_{\eta_k}\geq \delta_1$, it follows $I_{\eta_k+1}\in H_{\Delta}$ and $(S_{\eta_k+2}, I_{\eta_k+2})\in \mathcal U_{\delta_2}(s_*^+,i_*^+)$ provided $\sigma_{\eta_k+1}<T_2, \sigma_{\eta_k+2}>\bar T_2$. From \eqref{e3.14}we see that $(S_{\eta_k+2}, I_{\eta_k+2})\in \mathcal U_{\delta_2}(s_*^+,i_*^+)$ for infinitely many $k$. This means that $(s_*^+,i_*^+) \in \Omega(s_0, i_0, \omega)$ for almost all $\omega$.

Next, we show that $\{\pi_t^-(s_*^+,i_*^+):t \geq 0 \} \subset \Omega(s_0, i_0,\omega)$ a.s. Consider a point $(\bar s, \bar i) = \pi_{T_3}^-(s_*^+,i_*^+)$. By the continuous dependence of solutions on the initial values, for any ${\delta_4}>0$, there are $ \delta_3, \overline T_3$ such that if $(u,v)\in \mathcal U_{\delta_3}(s_*^+,i_*^+)$ then $\pi_t^-(u,v)  \in \mathcal U_{\delta_4}(\bar s,\bar i)$ for all $T_3-\overline T_3<t<T_3+\overline T_3.$ We now construct the sequence of stopping times 
\begin{eqnarray*}
\zeta_1&=&\inf\{2k+1: (S_{2k+1}, I_{2k+1})\in \mathcal U_{\delta_3}(s^+_*, i^+_*)\},\\
\zeta_2&=&\inf\{2k+1>\zeta_1: (S_{2k+1}, I_{2k+1})\in \mathcal U_{\delta_3}(s^+_*, i^+_*)\},\\
\cdots&&\\
\zeta_n&=&\inf\{2k+1>\zeta_{n-1}: (S_{2k+1}, I_{2k+1})\in \mathcal U_{\delta_3}(s^+_*, i^+_*)\}\ldots
\end{eqnarray*}

For $(s_*^+,i_*^+) \in \Omega(s_0, i_0, \omega)$, it follows that $\zeta_n<\infty$  and $\lim\limits_{n\to\infty}\zeta_n=\infty$ a.s. Since $\{\zeta_k=n\}\in\mathscr{F}_0^n$, $\{\zeta_k\}$ is independent of $\mathscr{F}_n^{\infty}.$ Put 
\begin{equation*}
B_k=\{\sigma _{\zeta _k+1}\in [T_3-\overline T_3,T_3+\overline T_3]\},\quad k=1,2,...
\end{equation*}
By the same argument as above we obtain ${\mathbb P}\{\omega: \sigma_{\zeta_n+1}\in[T_3-\overline T_3,T_3+\overline T_3] \mbox{ i.o. of \;}n\}=1.$  This implies $(S_{\zeta_k+1},I_{\zeta_k+1}) \in \mathcal U_{\delta_4}(\bar s, \bar i)$ for infinitely many times and $(\bar s, \bar i) \in \Omega(s_0, i_0,\omega) $ a.s. Thus, $\{\pi_t^-(s_*^+,i_*^+):t \geq 0 \} \subset \Omega(s_0, i_0,\omega)$. 

Based on the continuous dependence of solutions on the initial values and using a similar argument, we see that $\{\pi_{t_2}^+\circ\pi_{t_1}^-(s_*^+,i_*^+):t_1\geq 0, t_2 \geq 0 \} \subset \Omega(s_0, i_0,\omega)$. By induction, we conclude $\Gamma \subset \Omega(s_0, i_0,\omega)$. Moreover, $\overline \Gamma \subset \Omega(s_0, i_0,\omega)$ since $\Omega(s_0, i_0,\omega)$ is a closed set.
\end{proof}
\begin{proof}[The proof of Theorem \ref{thm3}]
We firstly point out the existence of a stationary distribution of the process $(\xi_t, S(t), I(t))$. 
From the proof of Proposition \ref{prop5}, we have $$\liminf_{t\to\infty}\frac 1t \int_0^t I(\bar t)d\bar t \geq \frac {c_{\min}}{(a_{\max}N+c_{\max})a_{\max}}\lambda=:\rho >0. $$
Denote by ${\bf1}_{A}$  the indicator function of the set $A$. By using  the relations
\begin{align*}
\dfrac{1}t\int^t_0I(\bar t)d\bar t=&\dfrac{1}t\int^t_0I(\bar t){\bf1}_{\{I(\bar t)<\frac{\rho}2\}}d\bar t+\dfrac{1}t\int^t_0I(\bar t){\bf1}_{\{I(\bar t)\geq\frac{\rho}2\}}d\bar t\\
\leq& \frac{\rho}2+\frac N t \int^t_0{\bf1}_{\{I(\bar t)\geq\frac{\rho}2\}}d\bar t,
\end{align*}
it follows, with probability 1,  that
$$
\liminf\limits_{t\to\infty}\dfrac{1}t\int_0^t{\bf1}_{\{I(\bar t)\geq\frac{\rho}2\}}d\bar t\geq  \dfrac{\rho}{2N}.
$$
Applying Fatou lemma yields 
\begin{equation}\label{bs}
\liminf\limits_{t\to\infty}\dfrac{1}t\int_0^t\mathbb P\big\{I(\bar t)\geq\frac{\rho}2\big\}d\bar t\geq  \frac{\rho}{2N}.
\end{equation} 
Consider the process $(\xi_t, S(t), I(t))$ on a larger state space $E\times \big(\nabla\setminus \{(s,i):s=0, 0\leq i\leq N\}\big)$.
it is easy to prove that $(\xi_t, S(t), I(t))$ is a Feller process. Therefore, 
by using  \cite[Theorem 4.5]{MT} (or \cite {SL}) 
the above estimate \eqref{bs} implies the existence of an invariant probability measure $\nu$ for the process $(\xi_t, S(t), I(t))$ on 
$E\times \big(\nabla\setminus \{(s,i):s=0, 0\leq i\leq N\}\big)$.
Since $\{(s,i): i=0, 0\leq s\leq N\}$ is invariant and $\lim_{t\to\infty}I(t)=0$ if $S(0)=0$, it follows that $\nu(\{(s,i):i=0, 0\leq s\leq N\})=0.$  Thus, $\nu(E\times(\nabla\cap\mbox{int}\mathbb R^2_+))>0$. By virtue of the invariant property of $E\times$int$\mathbb R^2_+$, the measure $\nu^*$  defined by $\nu^*(A)=\dfrac{\nu\big(A\cap E\times(\nabla\cap\mbox{int}\mathbb R^2_+)\big)}{\nu(E\times(\nabla\cap\mbox{int}\mathbb R^2_+)}$ for any measurable $A\in \mathcal{B}(\mathscr{V})$ is a stationary distribution on $E\times(\nabla\cap\mbox{int}\mathbb R^2_+)$ of the process $(\xi_t, S(t), I(t)).$
The absolute continuity of $\nu^*$ is proved as in the proof of \cite[Proposition 3.1]{DD} while the convergence in total variation of the distribution of $(\xi_t, S(t), I(t))$ can be referred to  \cite[Theorem 4.2]{DDY}.
\end{proof}

\medskip 
\noindent{\bf Acknowledgements} 

The Authors  would like to thank the reviewers for their very valuable remarks and comments, which will certainly improve the presentation of the paper.

\end{document}